\documentclass[12pt]{article}
\usepackage[utf8]{inputenc}
\usepackage[english]{babel}
\usepackage{amsthm}
\usepackage{hyperref}
\usepackage{cleveref}
\usepackage{amsmath, amssymb, mathrsfs}
\usepackage{xypic}
\usepackage{graphicx}
\usepackage{aliascnt}

\newtheorem{theorem}{Theorem}[section]
\newtheorem{corollary}[theorem]{Corollary}
\newtheorem{lemma}[theorem]{Lemma}
\theoremstyle{definition}
\newtheorem{definition}[theorem]{Definition}

\theoremstyle{remark}
\newtheorem{remark}[theorem]{Remark}

\theoremstyle{definition}
\newtheorem{example}[theorem]{Example}
\title{Local Complete Segal Spaces}
\author{Nicholas Meadows}
\begin{document}

\maketitle

\section{Introduction}
\label{intro}
The purpose of this paper is to develop a model structure on bi-simplicial presheaves in which the weak equivalences are stalkwise equivalences in the complete Segal model structure on bi-simplicial sets, and show that it is Quillen equivalent to the local Joyal model structure on simplicial presheaves of \cite{Nick}. The existence of the local complete Segal model structure was conjectured in \cite[Section 1.3]{Rezk}. The technique of Boolean localization is used extensively to develop this model structure (c.f. \cite{local} and \cite{Nick}).

This is the second in a series of three papers, including \cite{Nick} and \cite{Nick3}, which establish local analogues of three of the main extant models of higher category theory and establish a series of Quillen equivalences connecting them. The long-term objectives of this project are to apply these results to Simpson's theory of higher stacks, as discussed in \cite{Simpson-Descent}, and study variants of non-abelian cohomology (c.f. \cite[Section 5]{Nick3} for an initial thrust in this direction). As such, we have modelled our approach to local higher category theory on \cite{local}; this book contains numerous applications of local homotopy theory to geometric phenomena. 

In Section \ref{notation}, we establish notational and terminology conventions. In Section \ref{Segal} of the paper, we review some properties of the complete Segal model structure, as well as describe a variety of Quillen adjunctions between the complete Segal model structure, Joyal model structure, and standard model structure. These results are necessary for establishing the main results of the paper. We refer to \cite{Lurie} for facts about the Joyal model structure.  

In Section \ref{newmodel}, we define the local complete Segal model structure as the Bousfield localization of the Reedy model structure for bi-simplicial presheaves along the constant bisimplical presheaf maps $G(n) \subset F(n)$, $F(0) \subset I$. Using the technique of fibred sites (c.f. \cite{fibred-sites}), we can identify the Reedy model structure for bi-simplicial presheaves with the injective (Jardine) model structure on simplicial presheaves. Thus, we can use the localization theory of simplicial presheaves of \cite[Chapter 7]{local} to construct the local complete Segal model structure. This approach was chosen because it makes the fibrant objects, as well as descent, easy to describe (c.f. \ref{thm4.7}). 

In Section \ref{equiv}, we establish the main result of this paper: the Quillen equivalence between the local Joyal model structure and the local complete Segal model structure. 

In the Section \ref{descent}, we establish a result which relates descent in the local Joyal model structure to descent in the injective model structure. Interestingly, this result is proven using the Quillen equivalence established in Section \ref{equiv}. In addition, the description of the local complete Segal model structure as a Bousfield localization is key here.

\section{Notational Conventions}\label{notation}

For any category $\mathcal{C}$ we write $hom_{\mathcal{C}}(X, Y)$ for the set of morphisms between two $X, Y \in Ob(\mathcal{C})$. If the category is unambiguous we omit the subscript $\mathcal{C}$. We write $Iso(\mathcal{C})$ for the subcategory of $\mathcal{C}$ consisting of isomorphisms. If $\mathcal{C}$ is small, we write $B(\mathcal{C})$ for the nerve of the category. Given a simplicial set $K$, write $\pi(K)$ for the fundamental groupoid of $K$. Let $s\textbf{Set}$ denote the category of simplicial sets. Let $s^{2}\textbf{Set}$ denote the category of bi-simplicial sets.  We write $\textbf{hom}(X, Y)$ for the standard mapping complexes in both simplicial sets and bi-simplicial sets.

bi-simplicial sets are functors $X: \Delta^{op} \times \Delta^{op} \rightarrow Set$. We write $X_{m,n}$ for $X(m, n)$. We refer to $X_{m,n}$ as the $(m,n)$ bisimplices of $X$. Given simplicial sets $K$ and $L$, we can define a bi-simplicial set with $K \tilde{\times} L$ so that $(K \tilde{\times} L)_{m,n} = K_{m} \times L_{n}$. We write $\Delta^{p,q}$ for $\Delta^{p} \tilde{\times} \Delta^{q}$. 

In sections \ref{newmodel}- \ref{descent}, we fix a Grothendieck site $\mathscr{C}$. We denote the simplicial sheaves (respectively bi-simplicial sheaves) on $\mathscr{C}$ by $s\textbf{Sh}(\mathscr{C})$ (respectively $s^{2}\textbf{Sh}(\mathscr{C})$). We denote by $s\textbf{Pre}(\mathscr{C})$ (respectively $s^{2} \textbf{Pre}(\mathscr{C})$) the set of simplicial presheaves on $\mathscr{C}$ (respectively bi-simplicial presheaves). We also choose a Boolean localization $p: Sh(\mathscr{B}) \rightarrow Sh(\mathscr{C})$. Boolean localization is described in detail in \cite[Chapter 3 and 4]{local}. However, most of the facts we need are reviewed in \cite[Section 2]{Nick}. The \textbf{injective model structure} on $s\textbf{Pre}(\mathscr{C})$ is the standard model structure in which the cofibrations are monomorphisms and the weak equivalences are 'stalkwise' weak equivalences in the standard model structure on simplicial sets. We call its weak equivalences \textbf{local weak equivalences}. We call its fibrations \textbf{injective fibrations}. We define for a simplicial set $K$ a functor $hom(K, -) : s\textbf{Pre}(\mathscr{C}) \rightarrow \textbf{Pre}(\mathscr{C})$ by $hom(K, X)(U) = hom(K, X(U))$. For a simplicial set $K$ and a simplicial presheaf $X$, we write $X^{K}$ for the simplicial presheaf defined by $U \mapsto \textbf{hom}(K, X(U))$. Given a simplicial set $K$, we write $K$ for the constant simplicial presheaf $U \mapsto K$. The injective model structure has a function complex $\textbf{hom}(X, Y)$ with n-simplices defined by $\textbf{hom}(X, Y)_{n} = hom(X \times \Delta^{n}, Y)$. We also use the local Joyal model structure of \cite[Theorem 3.3]{Nick}. We call its weak equivalences \textbf{local Joyal equivalences} and its fibrations \textbf{quasi-injective fibrations}. We denote the sheafification functor for both simplicial and bi-simplicial presheaves by $L^{2}$.

\section{Complete Segal Spaces}\label{Segal}

\begin{definition}\label{def1.1}
Write $F(k) = \Delta^{k, 0} = \Delta^{k} \tilde{\times} \Delta^{0}$, and $\hat{F}(k) = \partial \Delta^{k} \tilde{\times} \Delta^{0}$. 
\end{definition}

\begin{definition}\label{def1.2}
Given a category C, its \textbf{discrete nerve}, $Disc(C)$, is defined to be the bi-simplicial set $B(C) \tilde{\times} \Delta^{0}$. We write $I = Disc(B\pi(\Delta^{1}))$. If $\textbf{n}$ is the ordinal number category, $Disc(\textbf{n})  = \Delta^{n} \tilde{\times} \Delta^{0} = F(n)$. Thus, there is a map $F(0) \rightarrow I$ induced by the inclusion of the initial vertex $\textbf{0} \subset \pi(\Delta^{1})$. 
\end{definition}

\begin{remark}\label{rmk1.3}
 We will identify $s\textbf{Set}$ with  a subcategory of $s^{2}\textbf{Set}$ via the embedding $K \mapsto \Delta^{0} \tilde{\times} K$. 
\end{remark}

\begin{definition}\label{def1.4}
Let $G(n)$ be the glued together string of 1-simplices $1 \le 2 \le \cdots \le n$ inside $\Delta^{n}$ regarded as a vertically discrete bi-simplicial set. Thus, there are natural inclusions $G(n) \subset F(n)$. 
\end{definition}

\begin{remark}\label{rmk1.5}
 Note that for a bi-simplicial set $X$, $\textbf{hom}(F(k), X) \cong X_{k}$, the vertical simplicial set in horizontal degree $k$ since  
\begin{equation*}
\begin{array}{rcl}
\textbf{hom}(\Delta^{k} \tilde{\times} \Delta^{0}, X)_{n} & \cong & hom(\Delta^{k} \tilde{\times} (\Delta^{0} \times \Delta^{n}), X)\\
& = & hom(\Delta^{k, n}, X ) \\
& \cong & X_{k, n}
\end{array}
\end{equation*}
Note that this implies that $\textbf{hom}(F(n), X) \rightarrow \textbf{hom}(G(n), X)$ can be identified with the map 
\begin{equation*}
X_{n*} \rightarrow X_{1*} \times_{X_{0*}} X_{1*} \cdots \times_{X_{0*}} X_{1*}  
\end{equation*}
where the right hand side is the limit of the diagram
\begin{equation*}
X_{1*} \xrightarrow{d_{1}} X_{0*} \xleftarrow{d_{0}} X_{1*} ....
\end{equation*}
constructed from n copies of $X_{1*}$. 
\end{remark}

\begin{example}\label{exam1.6}
The \textbf{Reedy model structure} on $s^{2}\textbf{Set}$ has cofibrations which are levelwise monomorphisms and weak equivalences which are levelwise weak equivalences. The generating cofibrations for the Reedy model structure are of the form

\begin{equation*} 
\partial(\Delta^{n} \tilde \times \Delta^{k}) = (\partial \Delta^{n} \tilde{\times} \Delta^{k}) \cup (\Delta^{n} \tilde{\times} \partial \Delta^{k}) \subset \Delta^{n} \tilde{\times} \Delta^{k}
\end{equation*} 
for $k, n \in \mathbb{N}$. The generating trivial cofibrations are of the form
\begin{equation*}
(\Delta^{k} \tilde{\times} \Delta^{n}) \cup (\Delta^{k} \tilde{\times} \Lambda_{r}^{n}) \subset \Delta^{k} \tilde{\times} \Delta^{n}
\end{equation*}
where $0 \le r \le n$.
\end{example}

\begin{definition}\label{def1.7}
 The \textbf{complete Segal model structure} is the left Bousfield localization of the Reedy model structure on $s^{2}\textbf{Set}$ along the set of maps  $ G(n) \subset F(n)$, $n \in \mathbb{N}$, and the natural inclusion $F(0) \rightarrow I$, where $I$ is as in \ref{def1.2}. The fibrant objects of this model category are called \textbf{complete Segal spaces}. 
\end{definition}

The complete Segal model structure first appeared in \cite{Rezk}.

\begin{example}\label{exam1.8}
If S is some set of maps in a simplicial model category, we say that $X$ is \textbf{S-local} if and only if X is fibrant and for each $g \in S, g^{*} : \textbf{hom}(D, X) \rightarrow \textbf{hom}(C, X)$ is a weak equivalence. By \cite[Theorem 4.1.1]{Hirschorn}, an object of $X$ is fibrant for the model structure of \ref{def1.7} if and only if it is fibrant in the Reedy model structure and it is S-local, where S is the set of maps in \ref{def1.7}. 
\end{example}

\begin{definition}\label{def1.9}
There are adjoint functors \begin{equation*}
k_{!} : s\textbf{Set} \leftrightarrows s\textbf{Set} : k^{!}
\end{equation*}
where 
$k_{!}(X) = \underset{\underset{\Delta^{n} \rightarrow X}{\longrightarrow}}{lim}(B\pi(\Delta^{n}))$
and $k^{!}(X)_{n} = hom(B\pi(\Delta^{n}), X)$
\end{definition}

The inclusion $\Delta^{n} \rightarrow B\pi(\Delta^{n})$ give a natural map $X \rightarrow k_{!}(X)$. Since $k^{!}(X)_{n} = hom(B\pi (\Delta^{n}), X)$, the inclusion $\Delta^{n} \rightarrow B\pi(\Delta^{n})$ induces natural maps $k^{!}(X)_{n} = hom(B \pi (\Delta^{n}), X) \rightarrow hom(\Delta^{n}, X) = X_{n}$, and hence induces a simplicial set map $k^{!}(X) \rightarrow X$.

\begin{lemma}\label{lem1.10}
The functor $k_{!}$ preserves monomorphisms and the natural map $X \rightarrow k_{!}(X)$ is a weak equivalence for simplicial sets. 
\end{lemma}

\begin{proof}
The fundamental groupoid functor takes pullbacks 
\begin{equation*}
\xymatrix{
\Delta^{n-2} \ar[r] \ar[d] & \Delta^{n-1} \ar[d]_{d_{i}} \\
\Delta^{n-1} \ar[r]_{d_{j}} & \Delta^{n}
}
\end{equation*}
to pullbacks. All maps $B\pi(\Delta^{n-1}) \rightarrow B\pi(\Delta^{n})$ are monomorphisms and 
there is a coequalizer diagram
\begin{equation*}
\coprod_{i < j} B\pi(\Delta^{n-2}) \rightarrow \coprod_{0 \le i \le n}B\pi(\Delta^{n-1}) \rightarrow C
\end{equation*}
where C is the union of the images in $B\pi(\Delta^{n})$. The functor $k_{!}$ preserves coequalizers so that $C \cong k_{!}(\partial \Delta^{n})$ and the induced map 
\begin{equation*}
k_{!}(\partial \Delta^{n}) \rightarrow k_{!}(\Delta^{n}) = B\pi(\Delta^{n})
\end{equation*}
is a monomorphism. The monomorphisms are the saturation of the inclusions $\partial \Delta^{n} \subset \Delta^{n}$. Since $k_{!}$ preserves colimits, it follows that $k_{!}$ preserves monomorphisms.

We show by induction on n that $X \rightarrow k_{!}(X)$ is a weak equivalence for all n-skeletal finite simplicial sets X. In the case $n=0$ this is trivial. In general, we can obtain $X$ as a finite succession of pushouts 
\begin{equation*}
\xymatrix{
\partial \Delta^{n} \ar[r] \ar[d] & Y \ar[d] \\
\Delta^{n} \ar[r] & Y'
}
\end{equation*}
where $Y \rightarrow k_{!}(Y)$ is a weak equivalence. By the inductive hypothesis $\partial \Delta^{n} \rightarrow k_{!}(\partial \Delta^{n})$ is a weak equivalence. Furthermore, $\Delta^{n} \rightarrow k_{!}(\Delta^{n}) = B\pi(\Delta^{n})$ is a weak equivalence. Thus, by the gluing lemma (\cite[Lemma 2.8.8]{GJ2}), we conclude that $Y' \rightarrow k_{!}(Y')$ is a weak equivalence. 

Let X be an infinite simplicial set. Let $M(X)$ be the set of finite subcomplexes of $X$. We have a commutative diagram
\begin{equation*}
\xymatrix{
\underset{\underset{K \in M(X)}{\longrightarrow}}{lim} K \ar[r] \ar[d]_{\cong}  & \underset{\underset{K \in M(X)}{\longrightarrow}}{lim} k_{!}(K) \ar[d]_{\cong} \\
X \ar[r]  & k_{!}(X)  
}
\end{equation*}
where the top horizontal map is a filtered colimit of weak equivalences. Since weak equivalences are preserved by filtered colimits, the map $X \rightarrow k_{!}(X)$ is a weak equivalence in general. 
\end{proof}

\begin{lemma}\label{lem1.11}
If $X$ is a Kan complex, the canonical map $k^{!}(X) \rightarrow X$ is a trivial Kan fibration of simplicial sets.
If $X$ is a quasi-category, the induced map $k^{!}(X) \rightarrow J(X)$ is a trivial fibration.
\end{lemma}

\begin{proof}
The lifting problem 
\begin{equation*}
\xymatrix{
\partial \Delta^{n} \ar[r] \ar[d] & k^{!}(X) \ar[d] \\
\Delta^{n} \ar@{.>}[ur] \ar[r] & X
}
\end{equation*}
is equivalent to a lifting problem 
\begin{equation*}
\xymatrix{
k_{!}(\partial \Delta^{n}) \cup_{\partial \Delta^{n}} \times \Delta^{n} \ar[r] \ar[d] & X \\
k_{!}(\Delta^{n} )\ar@{.>}[ur] & }
\end{equation*}
The diagram of monomorphisms
\begin{equation*}
\xymatrix{
\partial \Delta^{n} \ar[r]^{w.e.} \ar[d] & k_{!}(\partial \Delta^{n})  \ar[d] \ar@/^/[ddr]  & \\
\Delta^{n} \ar[r]^>>>>{w.e.} \ar[rrd]_{w.e.} & k_{!}(\partial \Delta^{n}) \cup_{\partial \Delta^{n}} \times \Delta^{n} \ar[dr] & \\
& & k_{!}(\Delta^{n})
}
\end{equation*}
 shows that $k_{!}(\partial \Delta^{n}) \cup_{\partial \Delta^{n}} \times \Delta^{n} \rightarrow k_{!}(\Delta^{n})$ is a trivial cofibration. Therefore, since X is a Kan complex, the required lift exists. 

For the second statement, note that every map $B\pi(\Delta^{n}) \rightarrow X$ factors through $J(X)$  by  \cite[Corollary 1.5]{Joyal1}. Thus, $k^{!}J(X) \rightarrow k^{!}(X)$ is an isomorphism. The induced map is the diagonal in the diagram 
\begin{equation*}
\xymatrix{
k^{!}(J(X)) \ar[r]_{\cong} \ar[d] & k^{!}(X) \ar[d] \ar[dl] \\ 
J(X)  \ar[r] & X
}
\end{equation*}
where $k^{!}J(X) \rightarrow J(X)$ is a trivial fibration by the first statement. 

\end{proof}

\begin{lemma}\label{lem1.12}
Suppose that $q: X \rightarrow Y$ is a quasi-fibration (i.e. a fibration in the Joyal model structure) and Y is a quasi-category. Then $k^{!}(q)$ is a Kan fibration. 
\end{lemma}

\begin{proof}
 All horn inclusions $\Lambda_{k}^{n} \rightarrow \Delta^{n}$ induce trivial cofibrations by \ref{lem1.10}.
Every diagram 
\begin{equation*}
\xymatrix{
k_{!}(\Lambda_{k}^{n}) \ar[r] \ar[d]_{i_{*}} & X \ar[d] \\
k_{!}(\Delta^{n}) \ar[r] & Y
}
\end{equation*}
can be refined to a diagram
\begin{equation*}
\xymatrix{
k_{!}(\Lambda_{k}^{n}) \ar[r] \ar[d]_{i_{*}} & J(X) \ar[d]_{J(q)} \ar[r] & X \ar[d]_{q} \\
k_{!}(\Delta^{n}) \ar[r] \ar@{.>}[ur] & J(Y) \ar[r] &  Y
}
\end{equation*}
Since $i_{*}$ is a trivial cofibration, to show that the lifting exists, it suffices to show that $J(q)$ is a Kan fibration (i.e. that $J$ takes quasi-fibrations to Kan fibrations).

Let $f$ be a quasi-fibration. By \cite[Corollary 1.3]{Joyal1}, $J$ preserves inner fibrations of quasi-categories. By the dual of \cite[Proposition 2.1.3.3]{Lurie}, it suffices to show that $J(f)$ is a right fibration, i.e. we want to solve liftings
\begin{equation*}
\xymatrix
{
\Lambda_{n}^{n} \ar[d] \ar[r] & J(X) \ar[d]_{J(f)} \\
\Delta^{n} \ar@{.>}[ur] \ar[r] & J(Y)
}
\end{equation*}
In the case $n = 1$ this follows from \cite[Corollary 2.4.6.5]{Lurie}. In the case $n \ge 2$, it follows from \cite[Remark 2.4.1.4]{Lurie} and \cite[Proposition 2.4.1.5]{Lurie}. 
\end{proof}

\begin{lemma}\label{lem1.13}
 The functor $k_{!}$ takes weak equivalences to Joyal equivalences. 
\end{lemma}

\begin{proof}
Suppose $Z$ is a quasi-category. Then $k^{!}(Z)$ is a Kan complex by \ref{lem1.12}. The functor $k^{!}$ preserves trivial fibrations, and takes quasi-fibrations between quasi-categories to Kan fibrations. 

Suppose that 
\begin{equation*}
\xymatrix{
& Z^{I} \ar[d] \\
Z \ar[r]_{\Delta} \ar[ur] & Z \times Z}
\end{equation*}
is a path object for the Joyal model structure. Then the induced diagram
\begin{equation*}
\xymatrix{
& k^{!}(Z^{I}) \ar[d]^{(p_{0*}, p_{1*})} \\
k^{!}(Z) \ar[r]_>>>>{\Delta} \ar[ur] & k^{!}(Z) \times k^{!}(Z)}
\end{equation*}
is a path object for the standard model structure. It follows that there are bijections 
\begin{equation*}
[X, k^{!}(Z)] \cong \pi(X, k^{!}(Z)) \cong \pi(k_{!}(X), Z) \cong [k_{!}(X),Z]
\end{equation*}
where $\pi(K, Y)$ is the set of right homotopy classes for the respective path objects constructed above. Therefore, $k_{!}$ takes weak equivalences to Joyal equivalences. 

\end{proof}

\begin{corollary}\label{cor1.14}
The adjoint pair 
$$
k_{!} : s\textbf{Set} \leftrightarrows s\textbf{Set} : k^{!}
$$
is a Quillen adjunction between the standard model structure on simplicial sets and the Joyal model structure.
\end{corollary}

\begin{proof}

Follows from \ref{lem1.10} and \ref{lem1.13}.
\end{proof}
The following theorem (\cite[Theorem 4.12]{JT1}) is a consequence of \ref{cor1.14} (see \cite[Sections 2-4]{JT1}).

\begin{theorem}\label{thm1.15}
  Let $t_{!}$ be the colimit-preserving functor defined by $t_{!}(\Delta^{n} \tilde{\times} \Delta^{m}) = \Delta^{n} \times B\pi(\Delta^{m})$. There is a Quillen equivalence
\begin{equation*}
t_{!} : s^{2} \textbf{Set} \leftrightarrows  s \textbf{Set} : t^{!}
\end{equation*}
between the complete Segal space model structure and the Joyal model structure.
\end{theorem}

\begin{example}\label{exam1.16}
Observe that 

\begin{equation*}
t^{!}(Y)_{m,n} \cong hom(\Delta^{m} \times B\pi(\Delta^{n}), X) \cong hom(B\pi(\Delta^{n}), \textbf{hom}(\Delta^{m}, X))
\end{equation*}
so that 
\begin{equation}\label{eq 1}
t^{!}(Y)_{m, *} =  k^{!}\textbf{hom}(\Delta^{m}, Y) 
\end{equation}

 A bi-simplicial set map $f: X \rightarrow t^{!}(Y)$ consists of maps 
\begin{equation*}
f: k_{!}(X_{m*}) \times \Delta^{m} \rightarrow Y
\end{equation*}
so that the diagrams 
\begin{equation*}
\xymatrix{
k_{!}(X_{n*}) \times \Delta^{m} \ar[d]_{\theta^{*} \times 1} \ar[r]_{1 \times \theta} & k_{!}(X_{n}) \times \Delta^{n} \ar[d] \\
k_{!}(X_{m*}) \times \Delta^{m} \ar[r]_{f} & Y}
\end{equation*}
commute for all ordinal number maps $\theta : [m] \rightarrow [n]$. It follows that
 
\begin{equation}\label{eq 2}
t_{!}(X) \cong d(k_{!}(X))
\end{equation}
\end{example}

\begin{lemma}\label{lem1.17}
Let K be a finite bi-simplicial set (i.e. having finitely many nondegenerate bisimplices) and $X \in s^{2}\textbf{Pre}(\mathscr{C})$. Then we have isomorphisms (natural in $K, X$)

\begin{enumerate}
\item{$p^{*}hom(K, X) \cong hom(K, p^{*}(X))$ if X is a simplicial sheaf }
\item{$p^{*}(X^{K}) \cong p^{*}(X)^{K} $ if X is a simplicial sheaf}
\item{$L^{2}hom(K, X) \cong hom(K, L^{2}(X))$ }
\item{$L^{2}(X^{K}) \cong L^{2}(X)^{K} $}
\end{enumerate}
 where $L^{2}$ denotes sheafification and $p$ is our choice of Boolean localization. 
\end{lemma}

\begin{example}\label{exam1.18}
Suppose that $X$ is a simplicial sheaf and $K$ is a simplicial set. 
Let $p : s\textbf{Sh}(\mathscr{B}) \rightarrow \textbf{Sh}(\mathscr{C})$ be a geometric morphism. We have isomorphisms 
\begin{equation*}
p_{*}hom(K, X) \cong \underset{\underset{\Delta^{n} \rightarrow K}{\longleftarrow}}{lim} p_{*}(X_{n}) \cong hom(K, p_{*}(X))
\end{equation*}
Recall that $k^{!}(X)_{m} = hom(B \pi(\Delta^{m}), X)$.
Thus, there is a natural isomorphism of sheaves
\begin{equation*}
p_{*}k^{!}(X) \cong k^{!}p_{*}(X)
\end{equation*}
Thus, by adjunction 
\begin{equation}\label{eq 3}
p^{*}L^{2}k_{!} \cong L^{2}k_{!}p^{*}L^{2} 
\end{equation}
\end{example}

\section{The Model Structure}\label{newmodel}

The following construction is an example of the Grothendieck construction for a presheaf of categories $A$ on a site $\mathscr{C}$.

\begin{definition}\label{def2.1}
There is a site $\mathscr{C}/A$ whose objects are all pairs $(U, x)$ where $U$ is an object of $\mathscr{C}$ and  $x \in Ob(A)(U)$. A morphism $(\alpha, f): (V, y) \rightarrow (U, x)$ in the category $\mathscr{C}/A$ is a pair consisting of a morphism $\alpha : V \rightarrow U$ of $\mathscr{C}$ along with a morphism $f: \alpha^{*}(x) \rightarrow y$ of $A(U)$. Given another morphism $(\gamma, g)$, the composite $(\alpha, f) \circ (\gamma, g)$ is defined by 
\begin{equation*}
(\alpha, f) \circ (\gamma, g) = (\alpha \gamma, g \cdot \gamma^{*}(f))
\end{equation*}
There  is a forgetful functor $c: \mathscr{C} /A \rightarrow \mathscr{C}$ which is defined by $(U, x) \mapsto U$. The covering sieves for $\mathscr{C}/A$ are the sieves which contain a sieve of the form $c^{-1}(S)$ for $S$ is a covering sieve of $\mathscr{C}$. 
\end{definition}

\begin{definition}\label{def2.2}
 Denote $s, t: Mor(A) \rightarrow Ob(A)$ the source and target maps. We will regard $Mor(A)$ and $Ob(A)$ as discrete simplicial presheaves.  An \textbf{$A$-diagram} is a simplicial presheaf map $\pi_{X} : X \rightarrow Ob(A)$ together with an 'action diagram'
\begin{equation*}
\xymatrix{
X\times_{s} Mor(A) \ar[d]_{pr} \ar[r]^<<<<<{m} &  X \ar[d]^{\pi_{X}} \\
Mor(A) \ar[r]_{t} & Ob(A)}
\end{equation*}
One further requires that $m$ respects compositions and identities. We denote by $s\textbf{Pre}(\mathscr{C})^{A}$ the category of A-diagrams whose morphisms are natural transformations 
\begin{equation*}
\xymatrix{
X \ar[rr]^{f} \ar[dr]_{\pi_{X}} & & Y \ar[dl]^{\pi_{Y}}\\
& Ob(A)& }
\end{equation*}
that respect compositions and identities. 
\end{definition}

\begin{example}\label{exam2.3}
There is a natural isomorphism of categories 
\begin{equation*}
sSet \cong Pre(*/ \Delta^{op})
\end{equation*}
Consequently, we have an identification 
\begin{equation*}
s^{2}\textbf{Pre}(\mathscr{C}) \cong s\textbf{Pre}(\mathscr{C}/ \Delta^{op})
\end{equation*}
\end{example}

\begin{theorem}\label{thm2.4}
(\cite[pg. 817-819]{fibred-sites}). Let $A$ be a presheaf of categories on $\mathscr{C}$. There is an equivalence of categories between $s\textbf{Pre}(\mathscr{C}/A)$ and $s\textbf{Pre}(\mathscr{C})^{A^{op}}$. This equivalence induces a model structure on $s\textbf{Pre}(\mathscr{C})^{A^{op}}$ defined as follows
\begin{enumerate}
\item{A weak equivalence (respectively a cofibration)
\begin{equation*}
\xymatrix{
X \ar[rr]^{f} \ar[dr] & & Y \ar[ld] \\
& Ob(A) & }
\end{equation*}
of $A^{op}$-diagrams is a map such that the simplicial presheaf map $f: X \rightarrow Y$ is a local weak equivalence (respectively monomorphism).
}
\item{A fibration of $A^{op}$-diagrams is a map which has the right lifting property with respect to all trivial cofibrations.}
\end{enumerate}
\end{theorem}

\begin{remark}\label{rmk2.5}
\ref{exam2.3} and \ref{thm2.4} imply that there is a Quillen equivalence 

\begin{equation*}
s\textbf{Pre}(\mathscr{C}/\Delta^{op}) \leftrightarrows s^{2}\textbf{Pre}(\mathscr{C})
\end{equation*}
 where the latter is equipped with a model structure in which a map $f: X \rightarrow Y$ is a weak equivalence (respectively cofibration) if and only if $X_{n*} \rightarrow Y_{n* }$ is a local weak equivalence (respectively monomorphism).

We call this model structure on bi-simplicial presheaves the \textbf{local Reedy model structure} and its weak equivalence \textbf{local Reedy equivalences}.  
\end{remark}

Suppose we choose a set $S$ of monomorphisms in $s\textbf{Pre}(\mathscr{C})$. By the results of \cite[Chapter 5]{local}, we can choose an uncountable regular cardinal $\alpha$ so that the $\alpha$-bounded cofibrations (respectively $\alpha$-bounded trivial cofibrations) form a set of generating cofibrations (respectively generating trivial cofibrations) for the injective model structure on $s\textbf{Pre}(\mathscr{C})$. We can form a smallest saturated set of monomorphisms $\mathcal{F}$, $S \subseteq \mathcal{F}$ subject to the following conditions 
\begin{enumerate}
\item{The class $\mathcal{F}$ contains all $\alpha$-bounded trivial cofibrations and all elements of $S$.}
\item{If $C \rightarrow D$ is an $\alpha$-bounded cofibration, and $A \rightarrow B$ is an element of $\mathcal{F}$, then $(A \times D) \cup (B \times C) \rightarrow B \times D$ is an element of $\mathcal{F}$.}
\end{enumerate}

The following is \cite[Theorem 7.18]{local}

\begin{theorem}\label{thm2.6}
 Let $\mathcal{F}$ be the set of cofibrations defined above. We call an object $X$ of $s\textbf{Pre}(\mathscr{C})$  $\mathcal{F}$\textbf{-injective} if the map $X \rightarrow *$ has the right lifting property with respect to each map in $\mathcal{F}$. We call a map a $\mathcal{F}-$\textbf{local equivalence} if and only if $\textbf{hom}(f, Z)$ is a weak equivalence of simplicial sets for each $\mathcal{F}$-injective object $Z$. There is a model structure on $s\textbf{Pre}(\mathscr{C})$, called the \textbf{$\mathcal{F}$-local model structure}, in which the weak equivalences are the $\mathcal{F}$-equivalences and cofibrations are monomorphisms.
\end{theorem}

Note that local weak equivalences are $\mathcal{F}$-equivalences.

\begin{lemma}\label{lem2.7}
 An $\mathcal{F}$-equivalence between two $\mathcal{F}$-injective objects of $s\textbf{Pre}(\mathscr{C})$ is a sectionwise weak equivalence.  

\end{lemma}

\begin{proof}
The $\mathcal{F}$-injective objects are the fibrant objects (\cite[Corollary 7.12]{local}), and a weak equivalence of fibrant objects is a simplicial homotopy equivalence. 
\end{proof}

\begin{definition}\label{def2.8}
Recall that we can identify bi-simplicial sets with constant bi-simplicial presheaves. Under this identification, let 
\begin{equation*}
S = \{ G(n) \subset F(n): n \in \mathbb{N} \} \cup \{ F(0) \subset I \}
\end{equation*}
Let $\mathcal{F}$ be the smallest saturated set containing $S$ as in \ref{thm2.6}.
Then the identification of \ref{rmk2.5} and \ref{thm2.6} applied to the family $\mathcal{F}$ give a model structure on $s^{2}\textbf{Pre}(\mathscr{C})$ called the \textbf{local complete Segal model structure}. We call its weak equivalences \textbf{local complete Segal equivalences}. We call its fibrations \textbf{Segal-injective} fibrations. 
\end{definition}

Let $U \in Ob(\mathscr{C})$. Then there exists a functor $L_{U} : s^{2}\textbf{Set} \rightarrow s^{2} \textbf{Pre}(\mathscr{C})$ defined by $L_{U}(K) = hom(-, U) \times K$.

\begin{remark}\label{rmk2.9}
Note that if $X$ is a fibrant object for the local complete Segal model structure, then it is a presheaf of complete Segal spaces. 

Indeed, $X$ has the right lifting property with respect to $L_{U}(i)$
where $i$ is one of the generating cofibrations for the Reedy model structure in \ref{exam1.6}. Thus, $X$ is sectionwise Reedy fibrant. 

Let $U \in Ob(\mathscr{C})$. Let $j_{n} : G(n) \rightarrow F(n)$ be the inclusion. By basic localization theory, 
$\textbf{hom}(L_{U}(j_{n}), X)$
is a weak equivalence for $n \in \mathbb{N}$. 
But this can be identified with $\textbf{hom}(F(n), X(U)) \rightarrow \textbf{hom}(G(n), X(U))$ (note that under the identification of \ref{exam2.3}, the constant simplicial presheaf $\Delta^{n}$ gets identified with the constant bi-simplicial presheaf $\Delta^{0} \tilde{\times} \Delta^{n}$). 

\end{remark}

\section{Equivalence with the local Joyal model Structure}\label{equiv}
Let 
$\mathcal{S}_{CSeg} : s^{2}\textbf{Pre}(\mathscr{C}) \rightarrow s^{2}\textbf{Pre}(\mathscr{C})$ and
$\mathcal{S}_{Joyal} : s\textbf{Pre}(\mathscr{C}) \rightarrow s\textbf{Pre}(\mathscr{C})$
denote, respectively, the functors obtained by applying the complete Segal and Joyal fibrant replacement functor sectionwise. Let $\mathcal{L}_{CSeg}, \mathcal{L}_{Joyal}, \mathcal{L}_{inj}$ denote, respectively, the fibrant replacement functors for the local complete Segal, local Joyal and injective model structures.

We define functors $t_{!} : s^{2}\textbf{Pre}(\mathscr{C}) \rightarrow s\textbf{Pre}(\mathscr{C})$ and $t^{!} : s\textbf{Pre}(\mathscr{C}) \rightarrow s^{2}\textbf{Pre}(\mathscr{C})$ by composition with $t_{!}$ and $t^{!}$ respectively. We also have functors $k_{!} : s\textbf{Pre}(\mathscr{C}) \rightarrow s\textbf{Pre}(\mathscr{C})$ and $k^{!} : s\textbf{Pre}(\mathscr{C}) \rightarrow s\textbf{Pre}(\mathscr{C})$.

\begin{lemma}\label{lem3.1}
There is a natural isomorphism $L^{2}t_{!}p^{*}L^{2} \cong p^{*}L^{2}t_{!}$.
\end{lemma}

\begin{proof}
This follows from equation \ref{eq 2} of \ref{exam1.16} and equation \ref{eq 3} of \ref{exam1.18}. 
\end{proof}

\begin{lemma}\label{lem3.2}
Let $f: X \rightarrow Y$ be a local weak equivalence. Then $k_{!}(f)$ is a local Joyal equivalence. 
\end{lemma}

\begin{proof}
Consider the natural sectionwise fibrant replacement map $\phi_{X} : X \rightarrow Ex^{\infty}(X)$. $k_{!}(\phi_{X})$ is a sectionwise, and hence local Joyal equivalence by \ref{cor1.14}. Thus, the diagram
\begin{equation*}
\xymatrix{
k_{!}(X) \ar[d]_{k_{!}(f)}\ar[r]_{} & k_{!}Ex^{\infty}(X) \ar[d]^{k_{!}Ex^{\infty}(f)} \\
k_{!}(Y) \ar[r] & k_{!}Ex^{\infty}(Y)}
\end{equation*}
and the 2 out of 3 property imply that we may assume that $f$ is a map of presheaves of Kan complexes. The fact that $p^{*}L^{2}$ preserves local weak equivalences, along with \cite[Lemma 4.23]{local}, imply that  $p^{*}L^{2}(f)$ is a sectionwise weak equivalence. Consider the diagram
\begin{equation*}
\xymatrix{
k_{!}p^{*}L^{2}(X) \ar[r] \ar[d]_{k_{!}p^{*}L^{2}(f)} & L^{2}k_{!}p^{*}L^{2}(X) \ar[d]^{L^{2}k_{!}p^{*}L^{2}(f)} \\
k_{!}p^{*}L^{2}(X) \ar[r] & L^{2}k_{!}p^{*}L^{2}(Y)
}
\end{equation*}
The left vertical map is a sectionwise, and hence local Joyal equivalence by \ref{cor1.14}. By \cite[Corollary 3.2]{Nick}, the horizontal maps are local Joyal equivalences. Thus, $L^{2}k_{!}p^{*}L^{2}(f) \cong p^{*}L^{2}k_{!}(f)$ is a local Joyal equivalence. But $p^{*}L^{2}$ reflects local Joyal equivalences by \cite[Remark 3.8]{Nick}. 
\end{proof}

\begin{lemma}\label{lem3.3}
Let $f: A \rightarrow B$ be a local Joyal equivalence and $g: C \rightarrow D$ be a cofibration. Then $h: A \times C \rightarrow B \times C$ and $u : (A \times D) \cup_{A \times C} (B \times C) \rightarrow B \times D$ are local Joyal equivalences. 
\end{lemma}

\begin{proof}
The second statement follows from left properness and the first statement. We prove the first statement. 

The map $A \times C \rightarrow \mathcal{S}_{Joyal}(A) \times \mathcal{S}_{Joyal}(C)$ is a sectionwise Joyal equivalence by \cite[Corollary 2.2.5.4]{Lurie} so it suffices to prove the statement for $A, B, C$ presheaves of quasi-categories. By \cite[Corollary 3.11]{Nick}, $p^{*}L^{2}(f)$ is a sectionwise Joyal equivalence. Thus, since $p^{*}L^{2}$ preserves finite limits, $p^{*}L^{2}(h)$ is isomorphic to 
\begin{equation*}
p^{*}L^{2}(A) \times p^{*}L^{2}(C) \rightarrow p^{*}L^{2}(B) \times p^{*}L^{2}(C)
\end{equation*}
which is a sectionwise Joyal equivalence by \cite[Corollary 2.2.5.4]{Lurie}. Thus, $h$ is a local Joyal equivalence, as required. 
\end{proof}

\begin{example}\label{exam3.4}
Recall that simplicial sets can be identified with constant simplicial presheaves. 
By a matching space argument, the generating trivial cofibrations for the local Reedy model structure on $s^{2}\textbf{Pre}(\mathscr{C})$ are of the form $f = (\Delta^{k} \tilde{\times} X) \cup (\partial \Delta^{k} \tilde{\times} Y) \rightarrow \Delta^{k} \tilde{\times} Y$, where $X \rightarrow Y$ is an $\alpha$-bounded trivial cofibration. 

Thus, since $t_{!}$ preserves colimits, we have

\begin{equation*}
t_{!}(f) =  (\Delta^{k} \times k_{!}(X)) \cup (\partial \Delta^{k} \times k_{!}(Y)) \rightarrow \Delta^{k} \times k_{!}(Y) 
\end{equation*}

The map $k_{!}(X) \rightarrow k_{!}(Y)$  is a local Joyal equivalence by \ref{lem3.2}. Thus, the map $t_{!}(f)$ is a local Joyal equivalence by \ref{lem3.3}.
\end{example}

\begin{lemma}\label{lem3.5}
Let $\mathcal{L}_{CSeg}$ be the fibrant replacement for the local complete Segal model structure. Then the natural map $t_{!}(X) \rightarrow t_{!}(\mathcal{L}_{CSeg}(X))$ is a local Joyal equivalence. 

\end{lemma}

\begin{proof}

Let $\mathcal{F}$ be the family defined in \ref{def2.8}. The fibrant objects of the local complete Segal model structure are the $\mathcal{F}$-injective objects by \cite[Corollary 7.12]{local}. Thus, $\mathcal{L}_{CSeg}$ is obtained by taking iterated pushouts along maps in a set $\mathcal{G}$ generating $\mathcal{F}$ (c.f. \cite[Lemma 10.21]{local}). 
The functor $t_{!}$ commutes with colimits, and filtered colimits of local Joyal equivalences are local Joyal equivalences. Thus, it suffices to show that $t_{!}(\phi)$ is a local Joyal equivalence where $\phi$ is in the diagram

\begin{equation*}
\xymatrix{
\coprod_{\mathcal{G}}Q \times  hom(Q, X) \ar[r] \ar[d] & X  \ar[d]_{\phi} \\
 \coprod_{\mathcal{G}} R \times hom(Q, X)  \ar[r] &  E_{1}(X)}
\end{equation*}
where $Q \rightarrow R$ is an element of $\mathcal{G}$. We can take $\mathcal{G}$ to be the set of maps $A \times D \cup B \times C \rightarrow B \times D$, where $C \rightarrow D$ is a $\alpha$-bounded cofibration and $A \rightarrow B$ is either
\begin{enumerate}
\item{$G(n) \subset F(n)$}
\item{$F(0) \rightarrow I $}
\item{A generating trivial cofibration for the local Reedy model structure}
\end{enumerate} 

Let $X$ be a complete Segal space. Then $\textbf{hom}(I \times D, X) \rightarrow \textbf{hom}(F(0) \times D, X)$ is naturally isomorphic to $\textbf{hom}(I, X^{D}) \rightarrow \textbf{hom}(F(0), X^{D})$. By \cite[Corollary 7.3]{Rezk}, $X^{D}$ is a complete Segal space. Since $F(0) \rightarrow I$ is a complete Segal equivalence, $\textbf{hom}(I, X^{D}) \rightarrow \textbf{hom}(F(0), X^{D})$ is a weak equivalence. It follows that 
$F(0) \times D \subset  I \times D$ is a complete Segal equivalence. Similarly, we can show that $G(n) \times D \subset F(n) \times D$ is a complete Segal equivalence. 

The functor $t_{!}$ takes sectionwise complete Segal equivalences to sectionwise Joyal equivalences by \cite[Theorem 4.12]{JT1}. The maps $t_{!}(F(0) \times D) \subset t_{!}(I \times D)$ and $t_{!}(F(0) \times D) \subset t_{!}(I \times D)$ are sectionwise Joyal equivalences, and hence local Joyal equivalences.  If $f$ a generating trivial cofibration for the local Reedy model structure, then $t_{!}(f \times id_{D})$ is a local Joyal equivalence by \ref{exam3.4} and \ref{lem3.3}. Thus $t_{!}(g)$, $g \in \mathcal{G}$, can be written as
\begin{equation*}
(t_{!}(A \times D)) \cup (t_{!}(B \times C)) \rightarrow t_{!}(B \times D)
\end{equation*}
The maps $t_{!}(A \times D) \rightarrow t_{!}(B \times D)$ and $t_{!}(A \times C) \rightarrow t_{!}(B \times C)$ are local Joyal trivial cofibrations by \ref{lem3.3} and \cite[Theorem 4.12]{JT1}. Thus, the map $t_{!}(g)$ is a local Joyal trivial cofibration. In conclusion, $t_{!}(\phi)$ is a pushout of a trivial cofibration for the local Joyal model structure, and is thus a trivial cofibration.
\end{proof}

\begin{lemma}\label{lem3.6}
$J$ preserves both trivial Kan fibrations and Kan fibrations.
\end{lemma}
\begin{proof}
Let $f: X \rightarrow Y$ be a Kan fibration. The map $f$ creates quasi-isomorphisms (i.e. 1-simplices that represent isomorphisms in the path category) since 
$\Delta^{1} \rightarrow B\pi \Delta^{1}$ is a trivial cofibration (c.f. \cite[Corollary 1.6]{Joyal1}). Thus, one has a pullback
\begin{equation*}
\xymatrix{
J(X) \ar[d] \ar[r] & J(Y) \ar[d] \\
X \ar[r] & Y 
}
\end{equation*}
 The same proof applies to trivial fibrations. 
\end{proof}

\begin{lemma}\label{cor3.7}
$J$ preserves local trivial fibrations. 
\end{lemma}

\begin{proof}

Let $f$ be a local trivial fibration. Then $p^{*}L^{2}(f)$ is a sectionwise trivial fibration so that $Jp^{*}L^{2}(f)$ is a sectionwise trivial fibration. But \cite[Lemma 3.6]{Nick} implies that $Jp^{*}L^{2}(f) \cong p^{*}L^{2}J(f)$. Thus, $J(f)$ is a local trivial fibration by \cite[Lemma 4.15]{local}.
\end{proof}

\begin{lemma}\label{lem3.8}
Let $f: X \rightarrow Y$ be a local Joyal equivalence of presheaves of quasi-categories. Then $t^{!}(f)$ is a local Reedy equivalence. 
\end{lemma}

\begin{proof}
By functorial factorization (\cite[Example 3.16]{Nick}), we can assume that $f$ is a sectionwise quasi-fibration (since $t^{!}$ preserves Joyal equivalences of quasi-categories). Thus, $f$ is a local trivial fibration by \cite[Lemma 3.15]{Nick}. Thus, so are the maps $f^{\Delta^{n}}$. By \ref{cor3.7}, each $J(f^{\Delta^{n}})$ is a local trivial fibration. But $J(f^{\Delta^{n}})$ is sectionwise Joyal equivalent to $t^{!}(f)_{n*} = k^{!}(f^{\Delta^{n}})$.  
\end{proof}

\begin{theorem}\label{thm3.9}
There is a Quillen equivalence 
\begin{equation*}
t_{!} : s^{2} \textbf{Pre} (\mathscr{C}) \leftrightarrows  s \textbf{Pre}(\mathscr{C}) : t^{!}
\end{equation*}
\end{theorem}

\begin{proof}
If X is a fibrant object of the local Joyal model structure then it is a presheaf of quasi-categories and $t_{!}t^{!}(X) \rightarrow X$ is a sectionwise Joyal equivalence by \cite[Theorem 4.12]{JT1} (note that every object is cofibrant in the model structures involved).

We want to show that the natural map $X \rightarrow t^{!}\mathcal{L}_{Joyal}t_{!}(X)$ is a local complete Segal equivalence. There is a commutative diagram 
\begin{equation*}
\xymatrix{
X \ar[r] \ar[d] & t^{!}\mathcal{L}_{Joyal}t_{!}(X) \ar[d] \\
\mathcal{L}_{CSeg}(X) \ar[r] & t^{!}\mathcal{L}_{Joyal}t_{!}\mathcal{L}_{CSeg}(X)\\
}
\end{equation*}
The map $\mathcal{L}_{Joyal}t_{!}(X) \rightarrow \mathcal{L}_{Joyal}t_{!}\mathcal{L}_{CSeg}(X)$ is a local Joyal equivalence of presheaves of quasi-injective objects by \ref{lem3.5}. Thus, it is a sectionwise Joyal equivalence. It follows from \cite[Theorem 4.12]{JT1} that the right vertical map is a sectionwise complete Segal equivalence of presheaves of complete Segal spaces. In particular, it is a sectionwise Reedy, and hence local complete Segal equivalence. The left vertical map is a local complete Segal equivalence by definition. Thus, we may assume that $X$ is a presheaf of complete Segal spaces.

The map $\mathcal{S}_{Joyal}t_{!}(X) \rightarrow \mathcal{L}_{Joyal}t_{!}(X)$ is a local Joyal equivalence of presheaves of quasi-categories. Thus, $t^{!}\mathcal{S}_{Joyal}t_{!}(X) \rightarrow t^{!}\mathcal{L}_{Joyal} t_{!}(X)$ is a local complete Segal equivalence by \ref{lem3.8}. By \cite[Theorem 4.12]{JT1}, the map $X \rightarrow t^{!}\mathcal{S}_{Joyal}t_{!}(X)$ is a sectionwise complete Segal equivalence. It is also a sectionwise Reedy equivalence (since it is a map of presheaves of complete Segal spaces), and hence a local complete Segal equivalence. It follows that the map 
\begin{equation*}
X \rightarrow t^{!}t_{!}(X) \rightarrow t^{!}\mathcal{S}_{Joyal}t_{!}(X) \rightarrow t^{!}\mathcal{L}_{Joyal}t_{!}(X)
\end{equation*}
is a local complete Segal equivalence, as required. 

\end{proof}

\begin{lemma}\label{lem3.10}
$t^{!}$ preserves and reflects local Joyal equivalences of presheaves of quasi-categories.
\end{lemma}
\begin{proof}
Consider the diagram
\begin{equation*}
\xymatrix
{
X \ar[r] \ar[d]_{f} & t_{!}t^{!} (X) \ar[r]_>>>>>{a} \ar[d]_{t_{!}t^{!}(f)} & t_{!}\mathcal{S}_{Joyal} t^{!} (X) \ar[d] \\
Y \ar[r] \ar[r] & t_{!}t^{!} (Y) \ar[r]_>>>>>{b} & t_{!}\mathcal{S}_{Joyal} t^{!} (Y)
}
\end{equation*}
The horizontal composites, $a$ and $b$ are all local Joyal equivalences. Thus, by 2 out of 3, the left horizonal maps are local Joyal equivalences.We conclude that $f$ is a local Joyal equivalence if and only if $t_{!}t^{!}(f)$ is a local Joyal equivalence. But $t_{!}$ preserves and reflects local Joyal equivalences. 
\end{proof}

\begin{corollary}\label{cor3.11}
A sectionwise complete Segal equivalence is a local complete Segal equivalence.
\end{corollary}

\begin{proof}
Let $f$ be a sectionwise complete Segal equivalence. Then $t_{!}(f)$ is a sectionwise Joyal equivalence, and hence a local Joyal equivalence. But $t_{!}$ reflects weak equivalences between cofibrant objects of the local Joyal model structure, as required. 
\end{proof}

\begin{corollary}\label{cor3.12}
 $p^{*}, L^{2}$ both preserve and reflect local complete Segal equivalences. 
\end{corollary}

\begin{proof}
Follows from the fact that $t_{!}$ preserves and reflects local equivalences, \ref{eq 1} and \ref{eq 3}. 
\end{proof}

\begin{theorem}\label{thm3.13}
The category $s^{2}\textbf{Sh}(\mathscr{C})$, along with the class of local complete Segal equivalences, monomorphisms and Segal-injective fibrations, forms a left proper model structure.
Let $i$ denote the inclusion of bi-simplicial sheaves into bi-simplicial presheaves. There is a Quillen equivalence

\begin{equation*}
L^{2}: s^{2}\textbf{Pre} (\mathscr{C}) \leftrightarrows  s^{2} \textbf{Sh}(\mathscr{C}) : i
\end{equation*}

\end{theorem}

\begin{proof}

The associated sheaf functor preserves and reflects local complete Segal equivalences and also preserves cofibrations. Hence, the inclusion functor preserves Segal-injective fibrations. Thus, the functors form a Quillen pair. The unit map of the adjunction $X \rightarrow L^{2}(X)$ is a local Reedy, and hence complete Segal equivalence, and the counit map is the identity. Thus, if we prove the first statement, we have the second.

 Axiom CM1 follows from completeness and cocompleteness of the sheaf category. Axioms CM2-CM4 follow from the corresponding statements for local complete Segal model structure on $s\textbf{Pre}(\mathscr{C})$.
By \cite[Lemma 7.4]{local}, there exists a regular cardinal $\alpha$ so that a map is a fibration in the complete Segal model structure if and only if it has the right lifting property with respect to $\alpha$-bounded trivial cofibrations. Choose a regular cardinal $\beta$ so that $L^{2}(f)$ is $\beta$ bounded for each $\alpha$-bounded $f$. Then a sheaf map $f$ is a Segal-injective fibration if and only if it has the right lifting property with respect to all $\beta$-bounded trivial cofibration. Doing a small object argument of size $2^{\beta}$, as in \cite[Lemma 5.7]{local}, gives CM5. 
\end{proof}

\begin{theorem}\label{thm3.14}
There is a Quillen equivalence 
\begin{equation*}
L^{2}t_{!}: s^{2} \textbf{Sh} (\mathscr{C}) \leftrightarrows  s \textbf{Sh}(\mathscr{C}) : t^{!}
\end{equation*}
\end{theorem}

\begin{proof}
Immediate from \ref{thm3.9}, and the fact that $t^{!}$ commutes with sheafification by equation \ref{eq 1}.  
\end{proof}

\section{Descent Results}\label{descent}

\begin{definition}\label{def4.1} 
One says that a simplicial presheaf (respectively bi-simplicial presheaf, respectively simplicial presheaf) $X$ \textbf{satisfies descent} for the injective (respectively local complete Segal, local Joyal) model structure if and only if $X \rightarrow \mathcal{L}_{inj}(X)$ (respectively $X \rightarrow \mathcal{L}_{CSeg}(X)$, $ X \rightarrow \mathcal{L}_{Joyal}(X)$) is a sectionwise weak equivalence (respectively sectionwise complete Segal equivalence, sectionwise Joyal equivalence). 
\end{definition}

\begin{lemma}\label{lem4.2}
Let $S$ be a simplicial set. $(-)^{S}$ preserves quasi-injective fibrations. 
\end{lemma}

\begin{proof}
Follows from \ref{lem3.3} since $(-)^{S}$ is right adjoint to $- \times S$. 
\end{proof}

\begin{lemma}\label{lem4.3}
Let $X$ be a fibrant object in the local Reedy model structure on $s^{2}\textbf{Pre}(\mathscr{C})$ (c.f. \ref{exam2.3}). Then $X_{n*}$ is a fibrant object in the injective model structure.
\end{lemma}

\begin{proof}
Consider the site morphism

\begin{equation*}
s_{n}: \mathscr{C} \cong \mathscr{C} / * \xrightarrow{n} \mathscr{C} / \Delta^{op}
\end{equation*}
where the latter map is inclusion of the nth vertex. By \cite[Corollary 5.24]{local}, the functor 
\begin{equation*}
(s_{n})_{*}: s\textbf{Pre}(\mathscr{C}/\Delta^{op}) \rightarrow s\textbf{Pre}(\mathscr{C})
\end{equation*}
is a right adjoint of a Quillen adjunction, and hence preserves fibrant objects. But $(s_{n})_{*}(X) = X_{n*}$. 
\end{proof}

\begin{lemma}\label{lem4.4}
If X is a presheaf of complete Segal spaces, then its local Reedy fibrant replacement (i.e. injective fibrant replacement under the identification of \ref{exam2.3}) $\mathcal{L}_{inj}(X)$ is Segal-injective fibrant. In particular, $X$ satisfies descent for the injective model structure if and only if it satisfies descent for the local complete Segal model structure. 
\end{lemma}
\begin{proof}
Consider the presheaf maps 
\begin{equation*}
\xymatrix{
X^{G(n)} \ar[r] \ar[d] & \mathcal{L}_{inj}(X)^{G(n)} \ar[d] \\
X^{F(n)} \ar[r] & \mathcal{L}_{inj}(X)^{F(n)}
}
\end{equation*}

\begin{equation*}
\xymatrix{
X^{I} \ar[r] \ar[d] & \mathcal{L}_{inj}(X)^{I} \ar[d] \\
X^{F(0)} \ar[r] & \mathcal{L}_{inj}(X)^{F(0)}
}
\end{equation*}
To show that $\mathcal{L}_{inj}(X)$ is Segal-injective fibrant, it suffices to show that the right vertical maps in the above diagram are local weak equivalences. 
The left vertical maps are sectionwise Reedy equivalences. The maps $X \rightarrow \mathcal{L}_{inj}(X)$ can be identified with a local weak equivalence of presheaves of Kan complexes. Since $(-)^{A}$ preserves local trivial fibrations, it preserves local weak equivalences of presheaves of Kan complexes by the functorial factorization of \cite[pg. 93]{local}. Thus, the horizontal maps in the above diagram are all local Reedy equivalences. Thus, by 2 out of 3, the right vertical maps are local weak equivalences, as required. 
\end{proof}

\begin{lemma}\label{lem4.5}
Let $X$ and $Y$ be presheaves of quasi-categories. A map $f: X \rightarrow Y$ is a local Joyal equivalence if and only if for all $n \in \mathbb{N}$
\begin{equation*}
J(X^{\Delta^{n}}) \rightarrow J(Y^{\Delta^{n}})
\end{equation*}
is a local weak equivalence.
\end{lemma}
\begin{proof}
If $X$ is a presheaf of quasi-categories, then so is each $X^{\Delta^{n}}$. Also, there is a sectionwise weak equivalence 
\begin{equation*}
k^{!}(X^{\Delta^{n}}) \rightarrow J(X^{\Delta^{n}})
\end{equation*}
Thus, the condition is equivalent to saying that $t^{!}(f)$ is a local Reedy equivalence. The result follows from \ref{lem3.8} and \ref{lem3.10}.

\end{proof}

\begin{lemma}\label{lem4.6}
Let $X$ be a presheaf of quasi-categories. Then $X$ satisfies descent with respect to the local Joyal model structure if and only if $t^{!}(X)$ satisfies descent with respect to the local complete Segal model structure.
\end{lemma}

\begin{proof}
The map $t^{!}(X) \rightarrow t^{!} \mathcal{L}_{Joyal}(X)$ is a local complete Segal equivalence, and $t^{!}\mathcal{L}_{Joyal}(X)$ is fibrant for the local complete Segal model structure. In particular, $t^{!} \mathcal{L}_{Joyal}(X)$ is a fibrant model of $t^{!}(X)$ in the local complete Segal model structure. The result follows from the fact that $t^{!}$ preserves and reflects sectionwise equivalence of presheaves of quasi-categories. 
\end{proof}

\begin{theorem}\label{thm4.7}
Let $X$ be a presheaf of quasi-categories. Then $X$ satisfies descent in the local Joyal model structure if and only if each $J(X^{\Delta^{n}})$ satisfies descent with respect to the injective model structure.  
\end{theorem}

\begin{proof}

If each $J(X^{\Delta^{n}})$ satisfies descent, then each $k^{!}(X^{\Delta^{n}})$ satisfies descent, because of the sectionwise weak equivalence $k^{!} \rightarrow J$. By \ref{lem4.3}, for $n \in \mathbb{N}$, $k^{!}(X^{\Delta^{n}}) = t^{!}(X)_{n*} \rightarrow \mathcal{L}_{CSeg} (t^{!}(X))_{n*}$ is an injective fibrant replacement (and a sectionwise weak equivalence). Therefore, $t^{!}(X)$ satisfies descent for the injective model structure. Conclude using \ref{lem4.4} and \ref{lem4.6}.

The proof of the converse is similar. 

\end{proof}

\begin{lemma}\label{lem4.8}
If $C$ is a category, then $JB(C) \cong B(Iso(C))$.
\end{lemma}

\begin{proof}
By construction, the n-simplices of $JB(C)$ are precisely the strings $a_{1} \rightarrow \cdots \rightarrow a_{n}$ of invertible arrows in $PB(C) \cong C$. 
\end{proof}

\begin{corollary}\label{cor4.9}
Let $C$ be a presheaf of categories. Then $B(C)$ satisfies descent for the local Joyal model structure if and only if for each $n \in \mathbb{N}$, $Iso(C)^{[\textbf{n}]}$ is a stack.
\end{corollary}

\begin{proof}
This follows from the preceding two results and the natural isomorphism $B(C)^{\Delta^{n}} = B(C)^{B([\textbf{n}])} \cong B(C^{[\textbf{n}]})$.
\end{proof}

\begin{theorem}\label{thm4.10}
Let $X$ be a presheaf of quasi-categories. Then one has a bijection
$[*, J(X)] = [*, X]_{q}$. Here, $[\, , \, ]_{q}$ denotes maps in the local Joyal homotopy category and $[\, , \,]$ denotes maps in the ordinary homotopy category on simplicial presheaves. 
 \end{theorem}

\begin{proof}
The constant simplicial presheaf $I = B \pi (\Delta^{1})$ is a interval object for the local Joyal model structure. Furthermore, every map $I \rightarrow X$ factors through $J(X)$ by \cite[Corollary 1.6]{Joyal1}. Since $\mathcal{L}_{Joyal}(X)$ satisfies descent, we have
\begin{equation*}
[*, X]_{q}  \cong [*, \mathcal{L}_{Joyal}(X)] \cong \pi_{I}(*, \mathcal{L}_{Joyal}(X)) \cong \pi_{I}(*, J\mathcal{L}_{Joyal}(X)) 
\end{equation*}
where $\pi_{I}(A, B)$ denotes the $I$-homotopy classes of maps. The presheaf $J\mathcal{L}_{Joyal}(X)$ satisfies descent with respect to the injective model structure. The constant simplicial presheaf map $\Delta^{1} \rightarrow I$ is a trivial cofibration in the injective model structure so we have
\begin{equation*}
\pi_{I}(*, J\mathcal{L}_{Joyal}(X))  \cong \pi_{\Delta^{1}}(*, J\mathcal{L}_{Joyal}(X)) \cong [*, J(X)] 
\end{equation*}
as required.
\end{proof}

\begin{example}\label{exam4.11}
If $A$ is a presheaf of categories, one has an identification $[*, BA]_{q} = [*, B(Iso(A))]$. In particular, \cite[Corollary 9.15]{local} implies that $[*, BA]_{q}$ is a non-abelian $H^{1}$ invariant. 
\end{example}
\bibliographystyle{amsplain}
\bibliography{database22}

\end{document}